\documentclass[12pt]{article}
\usepackage{amsmath,amsfonts,amssymb,amsthm,color}

\numberwithin{equation}{section}

\newcommand{\I}{{\bf 1}}
\newtheorem{proposition}{Proposition}[section]
\newtheorem{theorem}[proposition]{Theorem}
\newtheorem{corollary}[proposition]{Corollary}
\newtheorem{lemma}[proposition]{Lemma}
\newtheorem{remark}[proposition]{Remark}
\newtheorem{example}[proposition]{Example}
\newcommand{\nc}{\newcommand}
\nc{\R}{{\mathbb R}}
\nc{\N}{{\mathbb N}}
\nc{\Z}{{\mathbb Z}}

\nc{\BP}{\mathbb{P}}
\nc{\BE}{\mathbb{E}}
\nc{\BQ}{\mathbb{Q}}
\nc{\BX}{\mathbb{X}}
\nc{\bN}{{\mathbf N}}
\nc{\cB}{{\mathcal B}}
\nc{\cX}{{\mathcal X}}
\nc{\cC}{{\mathcal C}}
\nc{\cR}{{\mathcal R}}
\nc{\dint}{{\rm d}}

\DeclareMathOperator{\BV}{{\mathbb Var}}
\DeclareMathOperator{\CV}{{\mathbb Cov}}

\begin{document}

\author{G\"unter Last\footnotemark[1]\,, Mathew D.\ Penrose\footnotemark[2]\,,
Matthias Schulte\footnotemark[3]\, and Christoph Th\"ale\footnotemark[4]}

\title{Moments and central limit theorems for some\\
multivariate Poisson functionals}
\date{}
\renewcommand{\thefootnote}{\fnsymbol{footnote}}
\footnotetext[1]{Institute of Stochastics, Karlsruhe Institute of Technology,
76128 Karlsruhe, Germany. E-mail: guenter.last@kit.edu}

\footnotetext[2]{Department of Mathematical Sciences, University of Bath,
Bath BA2 7AY, United Kingdom. E-mail: m.d.penrose@bath.ac.uk}

\footnotetext[3]{Department of Mathematics, University of Osnabr\"uck, 49076 Osnabr\"uck, Germany. \textit{Current address:} Institute of Stochastics, Karlsruhe Institute of Technology, 76128 Karlsruhe, Germany. E-mail: matthias.schulte@kit.edu}

\footnotetext[4]{Department of Mathematics, University of Osnabr\"uck, 49076 Osnabr\"uck, Germany. \textit{Current address:} Faculty of Mathematics, Ruhr University Bochum, 44801 Bochum, Germany. E-mail: christoph.thaele@rub.de}

\maketitle

\begin{abstract}
\noindent
This paper deals with Poisson processes on an arbitrary measurable space. Using a direct approach, we derive formulae for moments and cumulants of a vector of multiple Wiener-It\^o integrals with respect to the compensated Poisson process. Second, a multivariate central limit theorem is shown for a vector whose components admit a finite chaos expansion of the type of a Poisson U-statistic. The approach is based on recent results of Peccati et al.\ combining Malliavin calculus and Stein's method, and also yields Berry-Esseen type bounds. As applications, moment formulae and central limit theorems for general geometric functionals of intersection processes associated with a stationary Poisson process of $k$-dimensional flats in $\R^d$ are discussed.
\end{abstract}


\vspace*{0.1cm}
\noindent
\textit{MSC (2010) subject classifications:} Primary: 60D05; 60H07 Secondary: 60F05; 60G55

\section{Introduction}\label{secintro}

Throughout this paper we denote by $({\Bbb X},{\cal X})$ a measurable space, equipped with a $\sigma$-finite measure $\lambda\neq 0$. A classical result by Wiener and It\^o says that if $F\equiv F(\eta)$ is a square integrable function of a Poisson process $\eta$ on $({\Bbb X},{\cal X})$ with intensity measure $\lambda$ (cf.\ \cite[Chapter 12]{Kallenberg}), then $F$ can be represented as an orthogonal $L^2$-series
\begin{equation}\label{chaos}
F={\Bbb E}F+\sum_{n=1}^\infty I_n(f_n),
\end{equation}
where $I_n(f_n)$ is the multiple ($n$-fold) Wiener-It\^o integral of a certain
symmetric function $f_n:{\Bbb X}^n\to{\Bbb R}$ that is square integrable with
respect to $\lambda^n$. For a proof of this result without any further assumptions
on the measure space $({\Bbb X},{\cal X},\lambda)$ we refer the
reader to \cite{LaPe11}. It turned
out that the \textit{chaos expansion} \eqref{chaos} is useful for many purposes.
For instance, it serves as mathematical basis for Malliavin calculus of
variations on the Poisson space and can be used to formulate and to prove
central limit theorems, see \cite{PSTU10,PeZheng10}.

The present paper deals with multivariate Poisson functionals $\big(F_t^{(1)},\ldots,F_t^{(\ell)}\big)$, $\ell\in{\Bbb N}$, where each component is of the form $F_t^{(i)}=F^{(i)}(\eta_t)$ and $\eta_t$ is a Poisson point process with an intensity measure of the form $\lambda_t=t\lambda$. Here, we are interested in the asymptotic regime that arises when the intensity parameter $t$ tends to infinity. Under the additional assumption that each $F_t^{(i)}$ is a \textit{U-statistic} of the Poisson process $\eta_t$, we prove formulae for the joint moments and cumulants and a multivariate central limit theorem.

The assumption that the functionals $F_t^{(i)}$ are Poisson U-statistics implies that their chaos expansions \eqref{chaos} terminate after finite numbers of terms, which is convenient for the application of Malliavin calculus. Univariate central limit theorems with bounds on the Wasserstein distance for Poisson functionals with finite Wiener-It\^o chaos expansions and, in particular, Poisson U-statistics are derived  in \cite{LRPeccati1,LRPeccati2,ReitznerSchulte2011} using a general result from \cite{PSTU10}. Our multivariate counterpart rests
on a multivariate analogue in \cite{PeZheng10}. By using a new truncation
argument and the special form of the Poisson functionals, our approach avoids
technical computations involving the product formula for multiple
Wiener-It\^o integrals that is used in the works mentioned previously.

As an application we study in detail geometric functionals of the intersection
process of order $m\in\{1,\ldots,d\}$ of a stationary Poisson $k$-flat process
in ${\Bbb R}^d$. We thereby considerably extend the results available in the
literature \cite{Hein09,HeinSS06} for the number of intersections
and the intersection volume. In our theory we can allow for very general
geometric functionals, for example, we do not require them to be
additive, translation-invariant or homogeneous. Furthermore, our central
limit theorems are quantitative in the sense that they provide rates of
convergence (with respect to a suitable distance).

For the asymptotic analysis of problems in stochastic geometry, two natural limiting regimes (among others) may be considered. On the one hand, one can fix the intensity of the underlying (Poisson) point process and increase the observation window in which everything takes place. On the other hand, one can fix this window and increase the intensity. We emphasize that these two regimes lead to limit theorems of different nature. Only in exceptional cases (such as for homogeneous functionals of Poisson $k$-flat processes considered in Section \ref{secflats}) it is possible to derive one limit theorem from the other. Our results deal with the situation of increasing intensity in case the functional of interest has the form of a Poisson U-statistic. As well as in the context described above, our theory can thus be applied directly to numbers of $k$-simplices of random simplical complexes \cite{Decreusefond} and to subgraph counting in random geometric graphs \cite{LRPeccati1,Penrose} with a fixed distance threshold. For problems that were previously considered in the literature for fixed intensity and increasing observation windows, such as the numbers of $k$-clusters \cite{BhattacharyaGhosh}, statistics of rather general random geometric graphs \cite{LRPeccati2}, or proximity functionals of non-intersecting $k$-flat processes \cite{ST13}, our results provide complementary central limit theorems for fixed windows and increasing intensity.

Another direction this paper deals with is formulae for mixed moments and cumulants, which in turn are based on identities for mixed moments and cumulants of multiple Wiener-It\^o integrals. We develop a quick approach to prove these formulae that have previously appeared in different generality in \cite{TaPecc09,Surg84}. The novelty of our proof is that it only makes use of elementary properties of the Poisson process (mainly the multivariate Mecke formula) and some combinatorial arguments, and deals directly with the expectation. In this way it avoids requiring the involved chaos expansion of such products.

The text is structured as follows. In Section \ref{secPreliminaries} we collect some basic definitions and background material. The moment formulae are presented in Section \ref{secmoments} while Section \ref{seclimits} deals with their asymptotic behaviour. Our multivariate central limit theorems for U-statistics of the underlying Poisson process are the content of Section \ref{seccentral} whereas in Section \ref{secflats} our results are applied to Poisson $k$-flat processes in ${\Bbb R}^d$.

\section{Preliminaries}\label{secPreliminaries}

In this paper all random objects are defined on a probability space $(\Omega,\mathcal{F},\BP)$. We interpret the Poisson process $\eta$ as a random element in the space $\bN:=\bN(\BX)$ of integer-valued (including $+\infty$) $\sigma$-finite measures $\mu$ on $\BX$ equipped with the smallest $\sigma$-field $\mathcal{N}$ making the mappings $\mu\mapsto\mu(B)$ measurable for all $B\in\mathcal{X}$. For $m\in\N$ and $\mu\in\bN$ we define a measure $\mu^{(m)}$ on $\BX^m$ by
\begin{align*}
\mu^{(m)}(B):=\idotsint\I_B(x_1,\ldots,x_m)
\Big(\mu - \sum_{j=1}^{m-1} \delta_{x_j}\Big)(\dint x_m)
\Big(\mu - \sum_{j=1}^{m-2} \delta_{x_j}\Big)(\dint x_{m-1}) \nonumber \\
\ldots (\mu - \delta_{x_1})(\dint x_2)\mu(\dint x_1),
\end{align*}
where $\delta_x$ is the Dirac measure located at a point $x\in\BX$. If $\mu$ is given as $\mu=\sum_I\delta_{x_i}$ for some countable index set $I$ and $x_i\in\BX$, $i\in I$, then
$$
\int f\,\dint\mu^{(m)}=\sideset{}{^{\ne}}\sum_{i_1,\ldots,i_m\in I} f(x_{i_1},\ldots,x_{i_m}),
$$
where $f$ is any non-negative measurable function on $\BX^m$ and where the superscript $\ne$ indicates that we sum over $m$-tuples of disjoint indices.

We will use the multivariate Mecke-formula (see e.g.\ \cite{LaPe11})
\begin{equation}\label{Meckem}
\begin{split}
\BE\int h(\eta,x_1,\ldots,x_m)\,&\eta^{(m)}\big(\dint(x_1,\ldots,x_m)\big)\\
&=\BE\int h(\eta+\delta_{x_1}+\ldots+\delta_{x_m},x_1,\ldots,x_m)\,
\lambda^m\big(\dint(x_1,\ldots,x_m)\big),
\end{split}
\end{equation}
which holds for all $h:\bN\times\BX^m\rightarrow\R$ for which one (and then also the other) side makes sense.

For any integer $n\ge 1$ let $L_s^1(\lambda^n)$ denote the set of all measurable and symmetric functions $f:\BX^n\rightarrow \R$ that are integrable with respect to $\lambda^n$. For $f\in L_s^1(\lambda^n)$ define the pathwise multiple Wiener-It\^o integral by
\begin{align}\label {prodIWI}
I_n(f):=\sum_{J \subset [n]}
(-1)^{n-|J|}\iint f(x_1,\dots,x_n)\eta^{(|J|)} (\dint x_J)
\lambda^{n-|J|}(\dint x_{J^c}),
\end{align}
where $[n]:=\{1,\dots,n\}$, $J^c:=[n]\setminus J$, $x_J:=(x_j)_{j\in J}$ and where $|J|$ denotes cardinality of $J$ (the inner integral is interpreted as $f(x_1,\ldots,x_n)$ in the case where $J=\emptyset$). By \eqref{Meckem}, $I_n(f)$ is a well-defined integrable random variable with $\BE I_n(f)=0$. If  $f\in L_s^1(\lambda^n)\cap L^2(\lambda^n)$, this pathwise definition coincides with the (classical)  definition of the multiple Wiener-It\^o integral for square integrable functions. (This fact can be derived, for instance, from equation (3.1) in \cite{LaPe11}.) The multiple Wiener-It\^o integral $I_n(f)$ for (symmetric) $f\in L^2(\lambda^n)$ is defined by an extension of the pathwise definition for $L^1$-functions to the space of all square integrable random variables. It has also mean zero and satisfies the orthogonality and isometry relations
\begin{align}\label{orth}
\BE I_m(g)I_n(h)=\I\{m=n\}\,n!\,\langle g , h \rangle_n,
\quad  m,n\ge 1,
\end{align}
for all (symmetric) $g\in L^2(\lambda^m)$ and $h\in L^2(\lambda^n)$, where $\langle\,\cdot\,,\,\cdot\,\rangle_n$ denotes the scalar product in $L^2(\lambda^n)$.

For $x\in\BX$ the {\em difference operator} $D_x$ is given as follows. For any measurable $F:\bN\rightarrow\R$ the function $D_xF$ on $\bN$ is defined by $D_xF(\mu):=F(\mu+\delta_x)-F(\mu)$, $\mu\in\bN$. For $n\ge 2$ and $(x_1,\ldots,x_n)\in\BX^n$ we define a function
$D^{n}_{x_1,\ldots,x_n}F:\bN\rightarrow\R$ by an iterated application of the difference operator $D$, that is, inductively by
\begin{align*}
D^{n}_{x_1,\ldots,x_{n}}F:=D^1_{x_{1}}D^{n-1}_{x_2,\ldots,x_{n}}F,
\end{align*}
where $D^1_x:=D_x$ and $D^0F:=F$. Under the assumption $\BE F(\eta)^2<\infty$ it was proved in \cite{LaPe11} that $D^{n}_{x_1,\ldots,x_{n}}F(\eta)$ is integrable for $\lambda^n$-a.e.\ $(x_1,\ldots,x_n)$ and that
\begin{align*}
T_n F(x_1,\ldots,x_n):=\BE D^n_{x_1,\ldots,x_n} F(\eta),
\quad (x_1,\ldots,x_n)\in\BX^n,
\end{align*}
defines a symmetric function in $L^2(\lambda^n)$. Moreover, we have the Wiener-It\^o chaos expansion
\begin{align}\label{chaos2}
F(\eta)={\Bbb E}F+\sum^\infty_{n=1}\frac{1}{n!}I_n(T_nF),
\end{align}
where the series converges in $L^2(\BP)$. Hence \eqref{chaos} holds with $f_n=\frac{1}{n!}T_nF$.

\section{Moments and cumulants}\label{secmoments}

Let $n\in\N$. A {\em subpartition} of $[n]$ is a family of disjoint and
non-empty subsets of $[n]$, which we call blocks. A {\em partition} of $[n]$
is a subpartition $\sigma$ of $[n]$ such that $\cup_{J\in\sigma}J=[n]$.
We denote by $\Pi_n$ (respectively $\Pi^*_n$) the system of all partitions
(respectively subpartitions) of $[n]$. The cardinality of
$\sigma\in \Pi_n^*$ (i.e.\ the number of blocks of $\sigma$)
is denoted by $|\sigma|$, while the cardinality of $\cup_{J\in\sigma}J$ is
denoted by $\|\sigma\|$. For any function $f:\BX^n\rightarrow\R$ and
$\sigma\in \Pi_n^*$ we define $f_\sigma:\BX^{|\sigma|+n-\|\sigma\|}\rightarrow\R$ by
identifying the arguments belonging to the same $J\in\sigma$.
(The arguments $x_1,\dots,x_{|\sigma|+n-\|\sigma\|}$ have to be inserted in the
order of occurrence.) In the case $n=4$ and $\sigma=\{\{2,3\},\{4\}\}$,
for instance, we have $f_\sigma(x_1,x_2,x_3)=f(x_1,x_2,x_2,x_3)$.

Consider $\ell,n_1,\dots,n_\ell\in\N$. Define $n:=n_1+\dots+n_\ell$ and
\begin{align}\label{Ji}
J_i:=\{j: n_1+\dots+n_{i-1}<j\le n_1+\dots+n_{i}\},\quad i=1,\dots,\ell.
\end{align}
Let $\pi:=\{J_i:1\le i\le \ell\}$ and let $\Pi(n_1,\dots,n_\ell)\subset \Pi_n$
(respectively $\Pi^*(n_1,\dots,n_\ell)\subset \Pi_n^*$) denote the set of all
$\sigma\in\Pi_n$ (respectively $\sigma\in\Pi_n^*$) with $|J\cap J'|\le 1$
for all $J\in\pi$ and all $J'\in \sigma$. Let $\Pi_{\ge 2}(n_1,\dots,n_\ell)$
(respectively $\Pi_{=2}(n_1,\dots,n_\ell)$) denote the set of all
$\sigma\in \Pi(n_1,\dots,n_\ell)$ with $|J|\ge 2$ (respectively $|J|=2$)
for all $J\in\sigma$. It is instructive to visualize the pair $(\pi,\sigma)$ as a
{\em diagram} with rows $J_1,\dots,J_\ell$, where the elements in each $J\in\sigma$
are encircled by a closed curve; see  \cite[Chapter 4]{TaPecc09} for more details
on such diagrams. Since the blocks of a $\sigma\in\Pi(n_1,\dots,n_\ell)$ are not
allowed to contain more than one entry from each row, the diagram $(\pi,\sigma)$
is called {\em non-flat} in \cite{TaPecc09}. Any $\sigma\in\Pi_{\ge 2}(n_1,\dots,n_\ell)$
induces a partition $\sigma^*\in\Pi_\ell$: $\sigma^*$ is the finest
partition of $[\ell]$ such that two numbers $i,j\in[\ell]$ are
in the same block of $\sigma^*$ if $J_i$ and $J_j$ are both intersected by
the same block of $\sigma$. Let $\tilde\Pi_{\ge 2}(n_1,\dots,n_\ell)$ be the
set of all $\sigma\in\Pi_{\ge 2}(n_1,\dots,n_\ell)$ such that $|\sigma^*|=1$.

The {\em tensor product} $\otimes_{i=1}^\ell f_i$ of functions
$f_i:\BX^{n_i}\rightarrow\R$, $i\in\{1,\dots,\ell\}$, is the function
from $\BX^n$ to $\R$ which maps each $(x_1,\dots,x_n)$ to $\prod_{i=1}^n f_i(x_{J_i})$.
In case that $f_1=\hdots=f_{\ell}=f$ we write $f^{\otimes \ell}$
instead of $\otimes_{i=1}^{\ell} f_i$.

The {\em joint cumulant} $\gamma\big(X_1,\dots,X_\ell\big)$ of $\ell\ge 1$
random variables $X_1,\dots,X_\ell$ is defined as
\begin{align*}
\gamma\big(X_1,\dots,X_\ell\big):=
(-\mathbf{i})^\ell\frac{\partial^\ell}{\partial z_1\dots \partial z_\ell}
\log\BE \Big[\exp[\mathbf{i}(z_1X_1+\dots+z_\ell X_\ell)]\Big]\Big|_{z_1=\dots =z_\ell=0},
\end{align*}
where $\mathbf{i}$ is the imaginary unit. This cumulant is well-defined if
$\prod_{j\in I} X_j$ is integrable for all $I\subset[\ell]$.
The $\ell$-th cumulant of a single random variable $X$ is defined by
$\gamma_\ell(X):=\gamma\big(X,\dots,X\big)$, where $X$ appears $\ell$ times.

The following result generalizes \cite[Corollary 7.4.1]{TaPecc09} and a
consequence of \cite{Surg84} to the case of more general Poisson processes and
integrands. In contrast to \cite{TaPecc09,Surg84}, we allow that the intensity
measure has atoms. Moreover, we avoid the assumption in \cite[Corollary 7.4.1]{TaPecc09} that the integrands are
simple functions. While the results in
\cite{TaPecc09,Surg84} are derived via formulae for the
Wiener-It\^o chaos expansion of a product of multiple Wiener-It\^o integrals,
here we take a direct approach, which relies only on \eqref{Meckem} and
some combinatorial arguments.

\begin{theorem}\label{thmoments} Let $f_i\in L_s^1(\lambda^{n_i})$,
$i=1,\dots,\ell$, where $\ell,n_1,\dots,n_\ell\in\N$. Assume that
\begin{align}\label{ass}
  \int (\otimes_{i=1}^\ell |f_i|)_\sigma \,
\dint\lambda^{|\sigma|}<\infty,\quad \sigma\in\Pi(n_1,\dots,n_\ell).
\end{align}
Then
\begin{align}\label{moments}
\BE \prod^\ell_{i=1}I_{n_i}(f_i)&=\sum_{\sigma\in\Pi_{\ge 2}(n_1,\dots,n_\ell)}
\int (\otimes_{i=1}^\ell f_i)_\sigma\, \dint\lambda^{|\sigma|},\\
\label{cumulants}
\gamma\big(I_{n_1}(f_1),\dots,I_{n_\ell}(f_\ell)\big)&=
\sum_{\sigma\in\tilde\Pi_{\ge 2}(n_1,\dots,n_\ell)}
\int (\otimes_{i=1}^\ell f_i)_\sigma\, \dint\lambda^{|\sigma|}.
\end{align}
\end{theorem}
\begin{proof} We abbreviate $f:=\otimes_{i=1}^\ell f_i$.
The definition \eqref{prodIWI} and Fubini's theorem imply that
\begin{equation}\label{12}
\begin{split}
\prod^\ell_{i=1}I_{n_i}(f_i)=
\sum_{I\subset[n]}(-1)^{n-|I|}&\idotsint f(x_1,\dots,x_n)\\
&\eta^{(|I\cap J_1|)}(\dint x_{I\cap J_1})\dots
\eta^{(|I\cap J_\ell|)}(\dint x_{I\cap J_\ell})
\lambda^{n-|I|}(\dint x_{I^c}),
\end{split}
\end{equation}
where $I^c:=[n]\setminus I$, and where we use definition \eqref{Ji} of $J_i$. For fixed $I\subset[n]$ we may split the
above integration according to $\sigma\in \Pi^*(n_1,\dots,n_\ell)$, where $\cup_{J\in\sigma}J=I$. For any such $\sigma$ we integrate (i.e.\ sum) over those $(x_1,\dots,x_n)$ satisfying $x_i=x_j$ whenever $i$ and $j$ belong to the same block of $\sigma$ but not otherwise. By \eqref{Meckem} applied with $h(\eta,y_1,\ldots,y_{|\sigma|})=f(x_1,\ldots,x_m)$ taking $x_i=y_j$ for $i$ in the $j$-th block of $\sigma$, the contribution of $\sigma$ to the expectation of the right-hand side of \eqref{12} equals $(-1)^{n-\|\sigma\|}\int f_\sigma\, \dint\lambda^{|\sigma|+n-\|\sigma\|}$. Therefore,
\begin{align}\label{m1}
\BE \prod^\ell_{i=1}I_{n_i}(f_i)=
\sum_{\sigma\in\Pi^*(n_1,\dots,n_\ell)}(-1)^{n-\|\sigma\|}
\int f_\sigma\, \dint\lambda^{|\sigma|+n-\|\sigma\|}.
\end{align}
By assumption \eqref{ass} all of these integrals are finite. Take a $\sigma\in\Pi^*(n_1,\dots,n_\ell)$ with $|J|\geq 2$ for all $J\in\sigma$ and consider the set $\Pi_1(\sigma)$ of all $\sigma_1\in\Pi^*(n_1,\dots,n_\ell)$ such that $\sigma\subset \sigma_1$ and
$|J|\le 1$ for all $J\in \sigma_1\setminus\sigma$. (Note that $\sigma\in\Pi_1(\sigma)$.) Observe that $\int f_{\sigma_1}\, \dint\lambda^{|\sigma_1|+n-\|\sigma_1\|}=\int f_{\sigma}\, \dint\lambda^{|\sigma|+n-\|\sigma\|}$ for all $\sigma_1\in\Pi_1(\sigma)$. Moreover, for $n-\|\sigma\|\ge 1$ we have that $$\sum_{\sigma_1\in\Pi_1(\sigma)}(-1)^{n-\|\sigma_1\|}=0.$$ Since every $\tau\in \Pi^*(n_1,\hdots,n_\ell)$ has a unique $\sigma\in\Pi^*(n_1,\hdots,n_\ell)$ with $|J|\geq 2$ for all $J\in\sigma$ such that $\tau\in\Pi_1(\sigma)$, we can partition $\Pi^*(n_1,\hdots,n_\ell)$ into the sets $\Pi_1(\sigma)$, $\sigma\in\Pi^*(n_1,\hdots,n_\ell)$ with $|J|\geq 2$ for all $J\in\sigma$. As shown above the sums over all $\sigma_1\in\Pi_1(\sigma)$ with $\|\sigma\|<n$ vanish and only the integrals related to partitions $\sigma\in \Pi_{\geq 2}(n_1,\hdots,n_\ell)$ remain. Therefore, \eqref{m1} implies the asserted identity \eqref{moments}.

We now prove \eqref{cumulants} by induction over $\ell$.  Since $\gamma\big(I_{n_1}(f_1),I_{n_2}(f_2)\big)=\BE I_{n_1}(f_1)I_{n_2}(f_2)$ and $\Pi_{\geq 2}(n_1,n_2)=\tilde{\Pi}_{\geq 2}(n_1,n_2)$, the identity \eqref{cumulants} is true for $\ell=2$, see \eqref{moments}. For $\ell\geq 3$, we obtain by the inversion formula expressing the $\ell$-th moment in terms of lower order cumulants (see e.g.\  \cite[Proposition 3.2.1]{TaPecc09}), formula \eqref{moments} and the assumption of the induction
\begin{equation}\label{inversion2}
\begin{split}
&\gamma\big(I_{n_1}(f_1),\hdots,I_{n_\ell}(f_\ell)\big)\\
 &=\BE \prod_{j=1}^\ell I_{n_j}(f_j)-\sum_{\pi\in\Pi_\ell,|\pi|>1}
\prod_{J\in\pi}\gamma\big((I_{n_j}(f_j))_{j\in J}\big)\\
&= \sum_{\sigma\in\Pi_{\geq 2}(n_1,\hdots,n_\ell)}\int f_{\sigma}\,\dint\lambda^{|\sigma|}
-\sum_{\pi\in\Pi_\ell,|\pi|>1}\prod_{J\in\pi}\sum_{\sigma_J\in\tilde{\Pi}_{\geq 2}(J)}
\int (\otimes_{j\in J}f_j)_{\sigma_J}\,\dint\lambda^{|\sigma_J|}.
\end{split}
\end{equation}
Here $\tilde{\Pi}_{\geq 2}(J)$ is defined in a similar way to $\tilde{\Pi}_{\geq 2}(n_1,\hdots,n_\ell)$. Now we use the fact that every partition $\sigma\in\Pi_{\geq 2}(n_1,\hdots,n_\ell)$ determines (in the obvious way) uniquely a partition $\pi=\sigma^*\in\Pi_\ell$ and a collection of partitions $\sigma_J\in\tilde{\Pi}_{\geq 2}(J), J\in\pi$, and vice versa. Combining this with Fubini's theorem, we have
\begin{eqnarray*}
\sum_{\pi\in\Pi_\ell,|\pi|>1}\prod_{J\in\pi}\sum_{\sigma_J\in\tilde{\Pi}_{\geq 2}(J)}
\int (\otimes_{j\in J}f_j)_{\sigma_J}\,\dint\lambda^{|\sigma_J|}
&=&\sum_{\sigma\in\Pi_{\geq 2}(n_1,\hdots,n_\ell),|\sigma^*|>1}
\int (\prod_{J\in\sigma^*}\otimes_{j\in J}f_j)_\sigma\,\dint\lambda^{|\sigma|}\\
&=&\sum_{\sigma\in\Pi_{\geq 2}(n_1,\hdots,n_\ell),|\sigma^*|>1}
\int f_\sigma\,\dint\lambda^{|\sigma|}.
\end{eqnarray*}
Hence in \eqref{inversion2} only the partitions $\sigma\in\Pi_{\geq 2}(n_1,\hdots,n_\ell)$ with $|\sigma^*|=1$ remain. In our notation, these are exactly the partitions in $\tilde{\Pi}_{\geq 2}(n_1,\hdots,n_\ell)$.
\end{proof}

\begin{remark}\label{remark:condition}\rm
The assumption \eqref{ass} is obviously satisfied if $f_i$ is bounded and $\lambda^{n_i}(\{f_i\neq 0\})<\infty$ for $i=1,\hdots,\ell$, which is the case for our examples in Section \ref{secflats}. But the assumption also holds under the following conditions.
Consider, as in Theorem \ref{thmoments}, measurable functions $f_i:\BX^{n_i}\rightarrow\R$, for $i=1,\dots,\ell$. Assume for any $i$ that $\{f_i\ne 0\}\subset B^{n_i}$, where $B\in\cX$ satisfies $\lambda(B)<\infty$. For any $i$, if $f_i\in L^\ell(\lambda^{n_i})$ then $f_i\in L^1(\lambda^{n_i})$ and \eqref{ass} holds. The second assertion  follows from the multivariate version of H\"older's inequality. In fact, if $\sigma\in\Pi(n_1,\dots,n_\ell)$, then
\begin{align*}
\left(\int (\otimes_{i=1}^\ell |f_i|)_\sigma \,\dint\lambda^{|\sigma|}\right)^\ell
\le \lambda(B)^{|\sigma|-n_1}\int |f_1|^\ell\, \dint\lambda^{n_1}
\cdot\ldots\cdot \lambda(B)^{|\sigma|-n_\ell}\int |f_\ell|^\ell\, \dint\lambda^{n_\ell}.
\end{align*}
Another sufficient condition for the assumptions of Theorem \ref{thmoments} is the existence of a function $g\in L^1(\lambda)\cap L^\ell(\lambda)$ such that $|f_i|\le g^{\otimes n_i}$ for any $i$. In this case we have for $\sigma\in\Pi(n_1,\dots,n_\ell)$ that
\begin{align*}
\int (\otimes_{i=1}^\ell |f_i|)_\sigma \,\dint\lambda^{|\sigma|}
\le \int g^{\otimes i_1}\,\dint\lambda \cdot\ldots\cdot \int g^{\otimes i_{|\sigma|}}\,\dint\lambda,
\end{align*}
where $i_1,\ldots,i_{|\sigma|}\le \ell$ are the cardinalities of the blocks of $\sigma$.
\end{remark}

\begin{example}\label{ex11}\rm Let $f\in L^1_s(\lambda^2)$ and consider
Theorem \ref{thmoments} in the case $\ell=2$, $n_1=n_2=2$ and $f_1=f_2=f$.
Then it is easy to see that assumption \eqref{ass} requires
$f\in L^2(\lambda^2)$ and
\begin{align*}
\int\bigg[\int |f(x_1,x_2)|\,\lambda(\dint x_2)\bigg]^2\,\,\lambda(\dint x_1)<\infty.
\end{align*}
Formula \eqref{moments} boils down to the isometry
relation $\BE  I_2(f)^2=2\langle f,f\rangle_2$. This shows that
assumption \eqref{ass} is not necessary for \eqref{moments}.
\end{example}

\begin{example}\label{ex12}\rm Let $f\in L^1_s(\lambda^2)$,
$g\in L^1(\lambda)$ and consider
Theorem \ref{thmoments} in the case $\ell=3$, $n_1=n_2=2$, $f_1=f_2=f$, $n_3=1$ and
$f_3=g$. Assumption \eqref{ass} then requires $f$ to satisfy the same
integrability conditions as in Example \ref{ex11}, as well as
\begin{align*}
&\int (|f(x_1,x_2)|+f(x_1,x_2)^2)
|g(x_1)|\,\lambda^2(\dint (x_1,x_2))<\infty,\\
&\int |f(x_1,x_2)f(x_2,x_3)|(|g(x_1)|+ |g(x_2)|)
\,\lambda^3\big(\dint (x_1,x_2,x_3)\big)<\infty.
\end{align*}
Formula \eqref{moments} means that
\begin{align*}
\BE[I_2(f)^2I_1(g)]=4\int f(x_1,x_2)^2g(x_1)\,\lambda^3\big(\dint (x_1,x_2)\big).
\end{align*}
Note that we do not need to assume $g$ to be square-integrable with respect
to $\lambda$.
\end{example}

\begin{corollary}\label{c2}
Let $f_n\in L^1_s(\lambda^n)$, $n\in\N$, and let $\ell\in\N$ and assume that
\begin{align}\label{ass22}
\int (\otimes^\ell_{i=1}|f_{n_i}|)_\sigma\,\dint\lambda^{|\sigma|}<\infty,
\quad \sigma\in\Pi(n_1,\dots,n_\ell),\, n_1,\dots,n_\ell\in\N.
\end{align}
Assume further that $\BE (\sum^\infty_{n=1}|I_n(f_n)|)^\ell<\infty$. Then the $\ell$-th moment and the $\ell$-th cumulant of
$F:=\sum^\infty_{n=1}I_n(f_n)$ are given by
\begin{align}\label{momchaos}
\BE F^\ell&=
\sum_{n_1,\dots,n_\ell\in\N}\;\sum_{\sigma\in\Pi_{\ge 2}(n_1,\dots,n_\ell)}
\int (\otimes_{i=1}^\ell f_{n_i})_\sigma\,\dint\lambda^{|\sigma|},\\
\label{cumchaos}
\gamma_\ell(F)&=
\sum_{n_1,\dots,n_\ell\in\N}\;\sum_{\sigma\in\tilde\Pi_{\ge 2}(n_1,\dots,n_\ell)}
\int (\otimes_{i=1}^\ell f_{n_i})_\sigma\,\dint\lambda^{|\sigma|}.
\end{align}
\end{corollary}
\begin{proof} Let $m\in\N$ and $F_m:=I_1(f_1)+\dots+I_m(f_m)$. Expanding $(F_m)^\ell$ and using \eqref{moments} gives
\begin{align*}
\BE F_m^\ell&=\sum^m_{n_1,\dots,n_\ell=1}
\BE (I_{n_1}(f_{n_1})\cdot\ldots\cdot I_{n_l}(f_{n_l}))\\
&=\sum^m_{n_1,\dots,n_\ell=1}\;\sum_{\sigma\in\Pi_{\ge 2}(n_1,\dots,n_\ell)}
\int (\otimes_{i=1}^\ell f_{n_i})_\sigma\,\dint\lambda^{|\sigma|}.
\end{align*}
Assumption $\BE (\sum^\infty_{n=1}|I_n(f_n)|)^\ell<\infty$ and dominated convergence imply \eqref{momchaos} for the infinite case. By the multilinearity of joint cumulants and \eqref{cumulants},
\begin{align*}
\gamma_\ell(F_m)&=
\sum^m_{n_1,\dots,n_\ell=1}\;\sum_{\sigma\in\tilde\Pi_{\ge 2}(n_1,\dots,n_\ell)}
\int (\otimes_{i=1}^\ell f_{n_i})_\sigma\,\dint\lambda^{|\sigma|}.
\end{align*}
Since $\gamma_\ell(F_m)$ is a polynomial in the moments $\BE F_m^j$, $j\in\{1,\dots,\ell\}$, whose coefficients are independent of $m$ (or $F_m$), we can again use dominated convergence to conclude the result \eqref{cumchaos} for the infinite case.
\end{proof}

\begin{remark}\rm If in Corollary \ref{c2} the number of non-vanishing functions $f_n\not\equiv 0$ is finite, then the assumption
$\BE (\sum^\infty_{n=1}|I_n(f_n)|)^\ell<\infty$ is implied by \eqref{ass22}.
\end{remark}

Let $f\in L^2(\lambda^2)$ be given by $f(x_1,x_2):=\I\{x_1\in B,x_2\in B\}$,
where $\lambda(B)<\infty$. By \eqref{prodIWI},
$$
I_2(f)=\eta(B)(\eta(B)-1)-2\eta(B)\lambda(B)+\lambda(B)^2.
$$
A straightforward calculation shows that
$\BE \exp(s I_2(f))=\infty$ for all $s>0$, whenever $\lambda(B)>0$.
Our next result shows that this is a quite general property
of Poisson functionals.

\begin{corollary}\label{c2a}
Let $F=\sum_{n=1}^\infty I_n(f_n)$ with $f_n\in L^1_s(\lambda^n)$, $n\in\N$, and assume that
$$\int (\otimes^\ell_{i=1}|f_{n_i}|)_\sigma \, \dint\lambda^{|\sigma|}<\infty\quad{\rm for\ all}\quad\sigma\in\Pi(n_1,\dots,n_\ell)\quad{\rm with}\quad n_1,\dots,n_\ell\in\N$$ and that  $\BE(\sum_{n=1}^\infty |I_n(f_n)|)^\ell<\infty$
for all $\ell\in\N$.
In addition, suppose that $f_n\geq 0$ for all $n\in\N$ and that there is an
$n_0\geq 2$ with $\|f_{n_0}\|_{n_0}>0$. Then
$\BE\exp(sF)=\infty$ for all $s>0$.
\end{corollary}

\begin{proof}
The idea of the proof is to show that the number of partitions we
sum over in the formulae \eqref{momchaos} and \eqref{cumchaos} is rapidly increasing in $\ell$. For $\ell\in\N$ with
$6|\ell$ (i.e.\ $\ell$ a multiple of $6$) let
$\Pi^{(2)}_\ell\subset\tilde{\Pi}_{\geq 2}(2,\hdots,2)$ be the set of partitions of $[2\ell]$
that can be constructed in the following way. First, the odd numbers in $[2\ell]$ are partitioned into
blocks of size $6$. Then we form $\ell/6-1$ blocks of size two from the even
numbers of $[2\ell]$ such that all partitions from $\Pi_{\geq 2}(2,\hdots,2)$ that contain the subpartition constructed so far must belong to $\tilde{\Pi}_{\geq 2}(2,\hdots,2)$.
Finally, we combine the remaining $\frac{2}{3}\ell+2$ even numbers of $[2\ell]$ into blocks of size two.
It follows from this construction and a short computation that
$$
|\Pi_{\ell}^{(2)}|\geq \frac{\ell!}{(\ell/6)! \, (6!)^{\frac{\ell}{6}}}
\frac{\left(\frac{2}{3}\ell+2\right)!}{(\frac{\ell}{3}+1)! \, 2^{\frac{\ell}{3}+1}}
\geq \frac{\ell! \, (\ell/6)!}{(6!)^\ell}.
$$
Note that we do not take into account here the different possibilities of forming the first $\ell/6-1$ blocks from the even numbers of $[2\ell]$. The previous inequality implies that
\begin{equation}\label{boundpartitions}
|\Pi_{\geq 2}(n_1,\hdots,n_\ell)|\geq |\tilde{\Pi}_{\geq 2}(n_1,\hdots,n_\ell)|
\geq |\Pi_{\ell}^{(2)}|\geq\frac{\ell! \, (\ell/6) !}{(6!)^\ell}
\end{equation}
if $n_1\geq 2,\hdots,n_\ell\geq 2$ and $\ell\in\N$ is such that $6|\ell$.

Each $\tilde{\sigma}\in\Pi^{(2)}_\ell$ with $\ell\in\N$ and $6|\ell$
induces for every $k\geq 2$ a partition $\sigma\in\tilde{\Pi}_{\geq 2}(k,\hdots,k)$
of $[k\ell]$ in the following way. First, one partitions $\{1+jk: j=0,\hdots,\ell-1\}$ as the odd
numbers in $\tilde{\sigma}$ and then for each $i=2,\hdots,k$ the sets $\{i+jk: j=0,\hdots,\ell-1\}$
as the even numbers in $\tilde{\sigma}$. We denote the set of these partitions
by $\Pi_{\ell}^{(k)}$.

Due to the assumptions that $\|f_{n_0}\|_{n_0}>0$ and that $f_{n_0}$ is non-negative and the structure of
$\Pi_\ell^{(n_0)}$ there must be a constant $c>0$ such that
\begin{equation}\label{boundintegrals}
\int (f_{n_0}^{\otimes \ell})_\sigma \, \dint\lambda^{|\sigma|} \geq c^{\ell}
\end{equation}
for all $\ell\in\N$ with $6|\ell$ and $\sigma\in\Pi^{(n_0)}_\ell$.
It follows from Corollary \ref{c2} and the estimates \eqref{boundpartitions}
and \eqref{boundintegrals} that
$$
\BE F^{\ell} \geq \frac{\ell!  \, (\ell/6)! }{(6!)^\ell}c^{\ell} \quad \text{ and }
\quad \gamma_{\ell}(F) \geq \frac{\ell! \, (\ell/6)! }{(6!)^\ell}c^{\ell}$$
for all $\ell\in\N$ with $6|\ell$. This implies that
$\BE\exp(s F)=\infty$ for all $s> 0$.
\end{proof}

\section{Asymptotic behaviour of moments and cumulants}\label{seclimits}

In this section we consider Poisson processes $\eta_t$ with intensity measures $\lambda_t:=t\lambda$, $t>0$. We are interested in functionals of $\eta_t$ that can be represented as $F_t=g(t)\sum f(x_1,\hdots,x_m)$ with the sum running over all $m$-tuples of distinct points of $\eta_t$ for some integer $m\geq 1$. This setting is taken from \cite[Section 5]{ReitznerSchulte2011}, where a central limit theorem for $F_t$ as $t\to\infty$ is derived. We generalize this to a multivariate setting and investigate the asymptotic behaviour of such functionals as $t\to\infty$. More formally, fix $\ell\geq 1$ and for $i=1,\dots,\ell$ let $m_i\in\N$, $f^{(i)}\in L^1_s(\lambda^{m_i})$ and $g_i:(0,\infty)\rightarrow\R$ such that $g_i(t)\ne 0$ for all (or at least for all sufficiently large) $t>0$. Now define
\begin{align}\label{Ustatistic}
F_t^{(i)}:=g_i(t)\int f^{(i)}(x_1,\dots,x_{m_i})\,\eta_t^{(m_{i})}\big(\dint(x_1,\dots,x_{m_i})\big), \quad t>0.
\end{align}
By \eqref{Meckem}, we have $$\BE F_t^{(i)}=g_i(t)t^{m_i} \int f^{(i)}(x_1,\hdots,x_{m_i})\,\lambda^{m_i}\big(\dint(x_1,\hdots,x_{m_i})\big).$$ For $n=1,\dots,m_i$, define
\begin{align}\label{fn}
f^{(i)}_n(x_1,\hdots,x_n):=\binom{m_i}{n}\int f^{(i)}(x_1,\hdots,x_n,y_1,\dots,y_{m_i-n})\,
\lambda^{m_i-n}\big(\dint(y_1,\hdots,y_{m_i-n})\big)
\end{align}
and denote by $I_{n,t}$ the $n$-fold Wiener-It\^o integral with respect to $\eta_t$. We claim that, $\BP$-almost surely, $F_t^{(i)}$ can be written as
\begin{align}\label{ch2}
F_t^{(i)}=\BE F_t^{(i)}+g_i(t) \sum^{m_i}_{n=1} t^{m_i-n}I_{n,t}(f^{(i)}_n).
\end{align}
Indeed, if $f^{(i)}_n\in L^2(\lambda^n)$ for all $n\leq m_i$ then $F_t^{(i)}$ is square-integrable and \eqref{ch2} is just a special case of the chaos expansion \eqref{chaos2}, cf.\ Lemma 3.5 in \cite{ReitznerSchulte2011}. The $L^1$-version can be derived by approximation or by a direct calculation (just plug \eqref{fn} into \eqref{prodIWI} and observe that all resulting terms cancel out,
except the integral representation \eqref{Ustatistic} of $F_t^{(i)}$).

Write $\|\,\cdot\,\|_n$ for the norm and $\langle\,\cdot\,,\,\cdot\,\rangle_n$ for the inner product in $L^2(\lambda^n)$, and assume again that $f^{(i)}_n\in L^2(\lambda^n)$ for $n\leq m_i$. Equations \eqref{orth} and \eqref{ch2} imply that
\begin{align}\label{var2}
&\BV[F^{(i)}_t]=g_i(t)^2\sum^{m_i}_{n=1}  t^{2m_i-2n}\,n!\,\int (f_n^{(i)})^2\, \dint\lambda_t^n
=g_i(t)^2\sum^{m_i}_{n=1}  t^{2m_i-n}\,n!\,\|f_n^{(i)}\|^2_n
\end{align}
and that
\begin{align}\label{cvij}
& \CV[F^{(i)}_t,F^{(j)}_t]=g_i(t)g_j(t)\sum^{\min\{m_i, m_j\}}_{n=1}
t^{m_i+m_j-n}\,n!\,\langle f^{(i)}_n,f^{(j)}_n\rangle_n.
\end{align}
The variances, covariances and mixed moments and cumulants of higher order show the following asymptotic behaviour as $t\to\infty$:

\begin{theorem}\label{thlimmom}
Assume that $\int (\otimes_{i=1}^\ell |f^{(i)}|)_{\sigma}\,\dint\lambda^{|\sigma|}<\infty$ for all $\sigma\in \Pi(m_1,\hdots,m_\ell)$. Then
\begin{align}\label{lmom}
\lim_{t\to\infty} \BE\prod_{i=1}^\ell \frac{(F_t^{(i)}-\BE F_t^{(i)})}{g_i(t)t^{m_i-1/2}}=
\sum_{\sigma\in\Pi_{=2}(1,\hdots,1)}\int (\otimes_{i=1}^\ell f_1^{(i)})_\sigma\,\dint\lambda^{|\sigma|}
\end{align}
and
\begin{align}\label{lcum}
  \lim_{t\to\infty}\frac{\gamma\big(F^{(1)}_t-\BE F^{(1)}_t,\hdots,F^{(\ell)}_t-\BE F^{(\ell)}_t\big)}
{\prod_{i=1}^\ell g_i(t)t^{m_i-1/2}}=0,\quad \ell\ge 3.
\end{align}
\end{theorem}

\begin{remark}\rm
Note that the right-hand side of \eqref{lmom} vanishes for odd $\ell$. Moreover, $\gamma\big(F_t^{(1)}-\BE F_t^{(1)}\big)=0$ and for $\ell=2$, the left-hand side of \eqref{lcum} coincides with that of \eqref{lmom} and equals $\langle f_1^{(1)},f_1^{(2)} \rangle_1$.
\end{remark}

\begin{proof}[Proof of Theorem \ref{thlimmom}.]
We can assume without loss of generality that $g_i(t)\equiv 1$. Due to the special structure of $F_t^{(i)}$ and the kernels of its chaos expansion, the integrability assumptions on $\otimes_{i=1}^\ell f^{(i)}$ imply that $\int (\otimes_{i=1}^\ell |f_{n_i}^{(i)}|)_\sigma\,\dint\lambda^{|\sigma|}<\infty$ for all $\sigma\in\Pi(n_1,\hdots,n_\ell)$ and $1\leq n_i\leq m_i$, $i=1,\hdots,\ell$. The latter is condition \eqref{ass22} for the functions $f_{n_i}^{(i)}$ in \eqref{ch2} and allows us to apply Theorem \ref{thmoments}, which yields
\begin{equation}\label{mixedmoments}
\begin{split}
\BE \prod_{i=1}^\ell(F^{(i)}_t-\BE F^{(i)}_t)&= \sum_{1\leq n_1\leq m_1,\hdots,1\leq n_\ell\leq m_\ell}
\BE\prod_{i=1}^\ell t^{m_i-n_i}I_{n_i,t}(f_{n_i}^{(i)})\\
&=\sum_{1\leq n_1\leq m_1,\hdots,1\leq n_\ell\leq m_\ell}
\;\sum_{\sigma\in\Pi_{\geq 2}(n_1,\hdots,n_\ell)}
\int (\otimes_{i=1}^{\ell} t^{m_i-n_i}f_{n_i}^{(i)})_{\sigma}\,\dint\lambda_t^{|\sigma|}.
\end{split}
\end{equation}
On the right-hand side, each summand has order $t^{\sum (m_i-n_i)+|\sigma|}$. Because of $|\sigma|\leq \lfloor (\sum n_i)/2 \rfloor$ and $\sum n_i\geq \ell$ we have $\sum (m_i-n_i)+|\sigma|\le \sum m_i-\lceil (\sum n_i)/2\rceil\leq \sum m_i-\lceil \ell/2\rceil$ so that the maximal order is at most $t^{\sum m_i-\lceil \ell/2\rceil}$. For even $\ell$ this is obtained if and only if $n_1=\hdots=n_\ell=1$, and the partition $\sigma$ satisfies $|J|=2$ for all $J\in\sigma$. Exactly these summands remain as $t\rightarrow\infty$ since they have the same order as the denominator in \eqref{lmom}; other summands vanish as $t\to\infty$. If $\ell$ is odd, the numerator has at most order $t^{\sum m_i-(\ell+1)/2}$ (in fact the order is attained) and the denominator has order $t^{\sum m_i-\ell/2}$ so that the expression vanishes in the limit.

For the cumulant $\gamma\big(F^{(1)}_t-\BE F^{(1)}_t,\hdots,F^{(\ell)}_t-\BE F^{(\ell)}_t\big)$, $\ell\geq 3$, we obtain by Theorem \ref{thmoments} the expression in the second line of \eqref{mixedmoments} where this time the inner sum only runs over all partitions $\sigma\in \tilde{\Pi}_{\geq 2}(n_1,\hdots,n_\ell)$. Since $\tilde{\Pi}_{\geq 2}(1,\hdots,1)\cap\Pi_{=2}(1,\hdots,1)=\emptyset$ for $\ell\geq 3$, all summands have a lower order than the denominator in \eqref{lcum} and vanish as $t\rightarrow\infty$.
\end{proof}

\bigskip

In the next result we take $F_t^{(1)}=\dots=F^{(\ell)}_t=F_t$ with
\begin{equation}\label{eq:Ftunivariate}
F_t:=g(t)\int f(x_1,\dots,x_{m})\,\eta_t^{(m)}\big(\dint(x_1,\dots,x_{m})\big), \qquad t>0,
\end{equation}
as in \eqref{Ustatistic}, where $g_1=\dots=g_{\ell}=g$, $f^{(1)}=\dots=f^{(\ell)}=f$ and $m_1=\dots=m_\ell=m$. Since $\Pi_{=2}(1,\hdots,1)$ has cardinality $$(\ell-1)!!:=(\ell-1)(\ell-3)\cdot\ldots \cdot 3\cdot 1$$ for even $\ell\ge 2$, Theorem \ref{thlimmom} implies the following result.

\begin{corollary}\label{c3}
Assume that $\int (|f|^{\otimes \ell})_{\sigma}\,\dint\lambda^{|\sigma|}<\infty$ for all $\sigma\in \Pi(m,\hdots,m)$ and that $\|f_1\|_1>0$. Then
\begin{align*}
\lim_{t\to\infty}\frac{\BE (F_t-\BE F_t)^\ell}{(\BV[F_t])^{\frac{\ell}{2}}}=
\begin{cases}
(\ell-1)!!, &\text{if $\ell$ is even},\\
0, &\text{if $\ell$ is odd},
\end{cases}
\end{align*}
and
\begin{align*}
  \lim_{t\to\infty}\gamma_\ell\left(\frac{F_t-\BE F_t}{\sqrt{\BV[F_t]}}\right)=
\begin{cases}
1, &\text{if $\ell=2$,}\\
0, &\text{if $\ell\ne 2$}.
\end{cases}
\end{align*}
\end{corollary}

\section{Central limit theorems}\label{seccentral}

In what follows, we assume the same setting as in the previous section. More precisely, fix $\ell\in{\Bbb N}$, let $F_t^{(1)},\hdots,F_t^{(\ell)}$ be defined as in \eqref{Ustatistic} and assume for each $i\le \ell$ that $f^{(i)}_n\in L^2(\lambda^n)$ for $n\leq m_i$. We shall at first show how the results of the previous section lead to a multivariate central limit theorem via the method of moments. Let us define
\begin{align}\label{def:fhat}
\hat{F}^{(i)}_t:=g_i(t)^{-1}t^{-(m_i-1/2)}(F^{(i)}_t-\BE F^{(i)}_t)
\end{align}
and note from \eqref{ch2} that
\begin{align}\label{4.33}
\hat{F}^{(i)}_t=t^{1/2}\sum^{m_i}_{n=1}t^{-n} I_{n,t}(f^{(i)}_n).
\end{align}
Furthermore, by \eqref{cvij}, we have the asymptotic covariances
\begin{align*}
  C_{ij}:=\lim\limits_{t\to\infty}\CV[\hat{F}^{(i)}_t,\hat{F}^{(j)}_t]=\langle f^{(i)}_1,f^{(j)}_1\rangle_1=\int f^{(i)}_1(x)f^{(j)}_1(x)\,\lambda(\dint x), \quad i,j\in\{1,\ldots,\ell\}.
\end{align*}

\begin{proposition}
Let $N$ be an $\ell$-dimensional centred Gaussian random vector with covariance matrix $(C_{ij})_{i,j=1,\dots,\ell}$ and assume that $\int(\otimes_{j=1}^k |f^{(i_j)}|)_\sigma \,\dint\lambda^{|\sigma|}<\infty$ for all $k\in\N$, $i_1,\hdots,i_k\in\{1,\hdots,\ell\}$ and $\sigma\in\Pi(m_{i_1},\ldots,m_{i_k})$. Then $\big(\hat F^{(1)}_t,\dots,\hat F^{(\ell)}_t\big)$ converges in distribution to $N$.
\end{proposition}
\begin{proof} Observe first that $\gamma(\hat F^{(i)}_t)={\Bbb E}\hat F_t^{(i)}=0$
for $1\leq i\leq\ell$ and
$$
\gamma(\hat F^{(i)}_t,\hat F^{(j)}_t)=\CV[\hat{F}^{(i)}_t,\hat{F}^{(j)}_t]
\to C_{ij}\qquad{\rm as}\qquad t\to\infty
$$
for any $1\leq i,j\leq\ell$.
Now fix integers $k\geq 3$ and $1\leq i_1\leq\ldots\leq i_k\leq\ell$,
and consider the joint cumulant $\gamma(\hat F^{(i_1)}_t,\dots,\hat F^{(i_k)}_t)$.
By homogeneity and \eqref{def:fhat} it follows that
$$
\gamma(\hat F^{(i_1)}_t,\dots,\hat F^{(i_k)}_t)
={\gamma(F_t^{(i_1)}-{\Bbb E}F_t^{(i_1)},\ldots,F_t^{(i_k)}
-{\Bbb E}F_t^{(i_k)})\over\prod_{j=1}^k g_{i_j}(t)t^{m_{i_j}-1/2}},
$$
whence Theorem \ref{thlimmom} implies that
$\gamma(\hat F^{(i_1)}_t,\dots,\hat F^{(i_k)}_t)\to 0$
as $t\to\infty$. The method of moments (or cumulants) now yields the
multivariate limit theorem, cf.\ \cite[p.\ 352]{Bil}. In the univariate case
the conclusion can also be directly drawn from Corollary \ref{c3}.
\end{proof}

\bigskip

We now turn to a quantitative version of the multivariate central limit theorem.
We measure the distance between two $\ell$-dimensional random vectors $X$ and $Y$ by
\begin{align}\label{d3}
d_3(X,Y):=\sup\limits_{g\in{\cal H}}|\BE g(X)-\BE g(Y)|,
\end{align}
where ${\cal H}$ is the set of all functions $h\in C^{3}({\R}^\ell)$ that satisfy
$$
\max\limits_{1\leq i_1\leq i_2\leq \ell}\,\sup_{x\in{\mathbb{R}}^\ell}
\left|\frac{\partial^2h(x)}{\partial x_{i_1}\partial x_{i_2}}\right|\le 1,
\quad
\max\limits_{1\leq i_1\leq i_2\leq i_3\leq \ell}\,
\sup_{x\in{\mathbb{R}}^\ell}\left|\frac{\partial^3h(x)}
{\partial x_{i_1}\partial x_{i_2}\partial x_{i_3}}\right|\leq 1.
$$
Note that convergence under the (pseudo-) metric $d_3$ implies convergence in distribution. In \cite{PeZheng10}, bounds are derived for the $d_3$-distance to the multivariate normal, along with similar bounds using a similarly defined $d_2$-distance. We work with the result for the $d_3$-distance since the covariance matrix of the Gaussian random vector is allowed to be only positive \textit{semi}-definite (this means that some linear combinations of the components of the limiting random vector may be constant). A non-trivial example for such a degenerate situation can be found in \cite{Hein09}. The multivariate normal approximation of Poisson U-statistics in the $d_2$-distance has been considered in \cite{Minh}.

In contrast to the univariate results for the Wasserstein distance discussed in the introduction, we can derive a multivariate result only for the $d_3$-metric since the underlying result in \cite{PeZheng10} is based on that distance.
This is caused by the fact that the approaches used in \cite{PeZheng10} for the multivariate normal approximation, namely an interpolation technique and the multivariate Stein's method, require a higher degree of smoothness for the test functions.

We are now ready to state the Berry-Esseen-type inequality.




\begin{theorem}\label{thmult} Let $N$ be an $\ell$-dimensional centered Gaussian random vector with covariance matrix $(C_{ij})_{i,j=1,\dots,\ell}$. Assume that $\int|f^{(i)}_1|^3\,\dint\lambda <\infty$ for every $i\in\{1,\dots,\ell\}$.
Then there is a constant $\tilde c>0$ such that
\begin{align*}
d_3\big((\hat F^{(1)}_t,\dots,\hat F^{(\ell)}_t),N\big)\le \tilde c t^{-1/2},\quad t\ge 1.
\end{align*}
\end{theorem}

\begin{remark}\label{rem:univariate}\rm
For $\ell=1$ it is possible to replace $d_3$ in Theorem \ref{thmult} by the classical Wasserstein distance $d_W$ and obtain that
$d_W\big(\hat F_t,N\big)\le c t^{-1/2}$, where $N$ is a standard Gaussian random variable with variance $\|f_1\|_1^2$ and $c$ is a constant (see \cite[Theorem 7.3]{LRPeccati2} or \cite[Theorem 5.2]{ReitznerSchulte2011} for a different rescaling). If $\|f_1\|_1=0$, this implies convergence in distribution to the constant random variable $N\equiv 0$. In this situation Theorem 7.3 in \cite{LRPeccati2} yields convergence in distribution to a higher-order Wiener-It\^o integral with respect to a Gaussian random measure after a suitable (different) scaling.
\end{remark}


\begin{remark}\rm Theorem \ref{thmult} also holds if $F^{(1)}_t,\dots,F_t^{(\ell)}$ are finite sums of random variables of type \eqref{Ustatistic}. In fact, under some additional conditions, any Poisson functional with finite Wiener-It\^o chaos expansion can be represented in such a way, cf.\ \cite{ReitznerSchulte2011}.
\end{remark}

We prepare the proof of Theorem \ref{thmult} by the following lemma:

\begin{lemma}\label{lem:Lemmad3}
Let $X$ and $Y$ be $\ell$-dimensional random vectors with $\BE X =\BE Y$ and Euclidean norms $||X||$ and $||Y||$ such that
$\BE||X||^2<\infty$, $\BE||Y||^2<\infty$. Then
$$
d_3(X,Y)\le \ell\sqrt{\BE||X||^2+\BE||Y||^2} \sqrt{\BE||X-Y||^2}.
$$
\end{lemma}
\begin{proof}
For $h\in {\cal H}$ and $X=(X_1,\hdots,X_\ell), Y=(Y_1,\hdots,Y_\ell)$, we obtain by the mean value theorem
\begin{align*}
\left|\BE h(X)-\BE h(Y)\right| =
\left|\BE[h'(Z) (X-Y)]-\BE [h'(0)(X-Y)]\right|,
\end{align*}
where $Z=Y+U(X-Y)$ for some random variable $U$ in $[0,1]$ and where we have used that the components of $X-Y$ all have expectation zero. Applying the mean value theorem again as well as the Cauchy-Schwarz inequality yields
\begin{align*}
\left|\BE h(X)-\BE h(Y)\right|
&=\left|\BE\sum_{i=1}^\ell\left(\frac{\partial h(Z)}{\partial u_i}-
\frac{\partial h(0)}{\partial u_i}\right)(X_i-Y_i)\right|\\
&=\left|\BE\sum_{i=1}^\ell\sum_{j=1}^\ell
\frac{\partial^2 h(\tilde{Z}^{(i)})}{\partial u_j\partial u_i}Z_j(X_i-Y_i)\right|\\
&\le \sqrt{\BE \sum_{i=1}^\ell\left(\sum_{j=1}^\ell\frac{\partial^2
h(\tilde{Z}^{(i)})}{\partial u_j\partial u_i}Z_j\right)^2}\sqrt{\BE||X-Y||^2}
\end{align*}
with random vectors $\tilde{Z}^{(i)}=U_i Z$ and random variables $U_i\in [0,1],$ $i=1,\dots,\ell$. By the fact that $h\in{\cal H}$ and the Cauchy-Schwarz inequality, it follows that
$$
\BE \sum_{i=1}^\ell\left(\sum_{j=1}^\ell\frac{\partial^2 h}{\partial u_j\partial u_i}
(\tilde{Z}^{(i)})Z_j\right)^2\leq \ell^2\,\BE||Z||^2\le \ell^2 \left(\BE||X||^2+\BE||Y||^2\right),
$$
which completes the argument.
\end{proof}

\begin{proof}[Proof of Theorem \ref{thmult}.]
For $i\in\{1,\dots,\ell\}$ define $$\bar{F}^{(i)}_t:=t^{-1/2}I_{1,t}(f_1^{(i)}),\quad t>0,$$ and note that $\CV[\bar{F}^{(i)}_t,\bar{F}^{(j)}_t]=C_{ij}$. Therefore we obtain from \cite[Corollary 4.3]{PeZheng10} that
\begin{align}\label{4.21}
d_3\big((\bar F^{(1)}_t,\dots,\bar F^{(\ell)}_t),N\big)\le
\frac{\ell^2}{4}\sum^\ell_{i=1}t^{-3/2}\int |f_1^{(i)}(x)|^3\,\lambda_t(\dint x)
=c_2 t^{-1/2}
\end{align}
for some $c_2>0$. Lemma \ref{lem:Lemmad3} implies
\begin{align}\label{4.22}
d_3\big((\hat F^{(1)}_t,\dots,\hat F^{(\ell)}_t),(\bar F^{(1)}_t,\dots,\bar F^{(\ell)}_t)\big)
\le A_t^{1/2}B_t^{1/2},
\end{align}
where
\begin{align*}
A_t:=\ell\sum_{i=1}^\ell \BE\big(\hat{F}^{(i)}_t\big)^2+\ell\sum_{i=1}^\ell\BE\big(\bar{F}^{(i)}_t\big)^2,
\qquad
B_t:=\sum_{i=1}^\ell \BE\big(\hat{F}^{(i)}_t-\bar{F}^{(i)}_t\big)^2.
\end{align*}
The first factor $A_t$ is bounded in $t$. For the second factor we use \eqref{4.33} to obtain that
\begin{align*}
B_t=\sum_{i=1}^\ell\BE\Big(\sum^{m_i}_{n=2}t^{-n+1/2} I_{n,t}(f^{(i)}_n)\Big)^2
=\sum_{i=1}^\ell\sum^{m_i}_{n=2}t^{-2n+1}t^{n} \|f^{(i)}_{n}\|^2_{n},
\end{align*}
so that $A_t^{1/2}B_t^{1/2}\le c_3 t^{-1/2}$, $t\ge 1$, for some $c_3>0$. Using this estimate in \eqref{4.22} and combining with
\eqref{4.21} and the triangle inequality for $d_3$, we obtain the result.
\end{proof}

\begin{remark}\rm
The proofs of Theorem \ref{thmult} and the univariate bound discussed in Remark \ref{rem:univariate} depend on general Berry-Esseen type inequalities for Poisson functionals from \cite{PSTU10,PeZheng10}, that are proven in a slightly more restrictive setting, namely that $(\BX,\cX)$ is a Borel space and $\mu$ is non-atomic. But they are still valid without these assumptions since the proofs only make use of properties of the Malliavin operators that also hold in our more general setting as shown in \cite{LaPe11}.
\end{remark}

\section{Poisson flat processes}\label{secflats}

In this section we assume that $\eta_t$ is a {\em stationary} Poisson process {\em of $k$-flats} ($k$-dimensional affine subspaces) in $\R^d$, where $d\ge 1$ and $k\in\{0,\ldots,d-1\}$. This is a Poisson process on the space $A(d,k)$ of all $k$-flats, whose distribution is invariant under translation of the flats. Its distribution is determined by the {\em intensity} $t>0$ and the {\em directional distribution} $\BQ$, a probability measure on the space $G(d,k)$ of all $k$-dimensional linear subspaces of $\R^d$. In fact, the intensity measure $\lambda_t$ of $\eta_t$ equals
\begin{align}\label{Lambda}
\lambda_t(\,\cdot\,)=t\int_{G(d,k)} \int_{E^\perp}
\I\{E+x\in\,\cdot\,\}\,{\cal H}^{d-k}(\dint x)\,\BQ(\dint E),
\end{align}
where ${\cal H}^{d-k}$ denotes $(d-k)$-dimensional Hausdorff measure, and $\BQ$ is a probability measure on $G(d,k)$. We let $\lambda:=\lambda_1$. If $\BQ$ is the uniform distribution (Haar measure), then $\eta_t$ is {\em isotropic}, that is, distributionally invariant under rotations. For further details on Poisson flat processes we refer to \cite{SW08}.

The {\em intersection process} of order $m\in\N$ is given as the set of all intersections $E_1\cap\dots \cap E_m$  of $m$ pairwise different flats in $\eta_t$. To introduce our geometric functionals of the latter process we let $\cC^d$ denote the system of all compact subsets of $\R^d$, equipped with the Borel $\sigma$-field induced by the Fell topology, see e.g.\ \cite[Chapter 12]{SW08}. We consider a measurable family $\cC^d_0\subset \cC^d$ of sets containing the empty set $\emptyset$ and with the property that $rB\cap E\in \cC^d_0$ for all $B\in\cC^d_0$, all $r>0$, and all affine subspaces $E\subset\R^d$. We assume that $\psi:\cC^d_0\rightarrow\R$ is a measurable function with $\psi(\emptyset)=0$ satisfying
\begin{align}\label{boundedweak}
 \int |\psi(B\cap E_1\cap\dots \cap E_m)|^3\,
\lambda^{m}\big(\dint(E_1,\dots,E_m)\big) \le C_B
\end{align}
for all $B\in\cC^d_0$, where $C_B\ge 0$ is a constant only depending on $B$. (By \cite[Theorem 12.2.6]{SW08} the mapping  $(E_1,\dots,E_m)\mapsto B\cap E_1\cap\dots \cap E_m$ is measurable.) We note here that \eqref{Lambda} implies that $\lambda$ is {\em
locally finite} in the sense that
$\lambda(\{E\in A(d,k):B\cap E\ne\emptyset\})<\infty$ for all
$B\in\cC^d$.
Since $\psi(\emptyset)=0$, assumption \eqref{boundedweak} implies
the integrability of $|\psi(B\cap E_1\cap\dots \cap E_m)|^p$ w.r.t.\
$\lambda^m$ for any $p\in(0,3]$.
This is enough to settle all integrability issues in this section.
Clearly \eqref{boundedweak} is implied by the stronger condition
\begin{align}\label{bounded}
  |\psi(B\cap E_1\cap\dots \cap E_m)|\le c_B, \quad
\lambda^m\text{-a.e.\ $(E_1,\dots,E_m)$}, B\in\cC_0^d
\end{align}
for some $c_B\ge 0$ depending on $B$. In particular, \eqref{bounded} is satisfied in our examples below.

Define a random field $\zeta_t:=\{\zeta_t(B):B\in \cC^d_0\} $ by
\begin{align*}
  \zeta_t(B):=\frac{1}{m!}\int \psi(B\cap E_1\cap\dots \cap E_m)\,
\eta_t^{(m)}\big(\dint(E_1,\ldots,E_m)\big),\qquad B\in\cC^d_0.
\end{align*}
Since $\eta_t$ has only atoms of size one (by \eqref{Lambda}) we can identify $\eta_t$ with its support, and integration with respect to $\eta_t^{(m)}$ corresponds to summation over all $m$-tuples $(E_1,\dots,E_m)\in\eta_t^m$ with pairwise different entries. For $A,B\in\cC^d_0$ define
\begin{equation}\label{cov}
\begin{split}
C(A,B):=&\frac{1}{((m-1)!)^2}
\int\Big(\int \psi(A\cap E_1\cap E_2\cap \dots \cap E_m)\,
\lambda^{m-1}\big(\dint(E_2,\dots,E_m)\big)\Big)\\
&\times\Big(\int \psi(B\cap E_1\cap E'_2\cap \dots \cap E'_m)\,
\lambda^{m-1}\big(\dint(E'_2,\dots, E'_m)\big)\Big)\,\lambda(\dint E_1).
\end{split}
\end{equation}
If $m=1$, this has to be read as
\begin{align*}
C(A,B)=
\int \psi(A\cap E_1)\psi(B\cap E_1)\,\lambda(\dint E_1).
\end{align*}
It can be checked directly that $C(\cdot,\cdot)$ is positive semidefinite. Therefore we can consider a centred Gaussian field $\xi:=\{\xi(B):B\in \cC^d_0\}$ with this covariance function.

Define
\begin{align*}
\hat\zeta_t(B):=t^{-(m-1/2)}(\zeta_t(B)-\BE \zeta_t(B)),\quad t>0,B\in\cC^d_0.
\end{align*}

\begin{theorem}\label{cltflats}
Let $\ell\ge 1$ and $B_1,\dots,B_\ell\in\cC^d_0$. Then
\begin{align*}
  d_3\big((\hat\zeta_t(B_1),\dots,\hat\zeta_t(B_\ell)),(\xi(B_1),\dots,\xi(B_\ell))\big)\le
c(B_1,\dots,B_\ell)t^{-1/2},\quad t\ge 1,
\end{align*}
for some constant $c(B_1,\dots,B_\ell)$. In particular
\begin{align*}
\{\hat\zeta_t(B):B\in \cC^d_0\}
\overset{d}{\longrightarrow} \{\xi(B):B\in \cC^d_0\} \quad \text{as $t\to\infty$}
\end{align*}
in the sense of convergence of finite-dimensional distributions.
\end{theorem}
\begin{proof} This is a direct consequence of Theorem \ref{thmult}.\end{proof}

\bigskip

Alternatively one can approach the central limit problem in another but closely related setting. Instead of increasing the intensity parameter $t$, we can also fix $t$ (for simplicity we take $t=1$) and increase the size $r$ of the observation window. If we assume additionally that the considered function $\psi$ is homogeneous of degree $\alpha\in\R$, that is
\begin{align}\label{homo}
\psi(rB)=r^{\alpha}\psi(B),\quad B\in\cC^d_0,\, r>0,
\end{align}
both approaches are equivalent. Define a random field $\tilde\zeta_r:=\{\tilde\zeta_r(B): B\in\cC_0^d\}$ with $\tilde\zeta_r(B)=r^{-(m-1/2)(d-k)-\alpha}(\zeta_1(rB)-\BE\zeta_1(rB))$.

\begin{corollary}\label{cltflats2}
Assume \eqref{homo}, let $\ell\ge 1$ and $B_1,\dots,B_\ell\in\cC^d_0$. Then
\begin{align*}
  d_3\big((\tilde\zeta_r(B_1),\dots,\tilde\zeta_r(B_\ell)),(\xi(B_1),\dots,\xi(B_\ell)\big)\le
c(B_1,\dots,B_\ell)r^{-(d-k)/2},\quad r\ge 1,
\end{align*}
for some constant $c(B_1,\dots,B_\ell)$. In particular
\begin{align*}
\{\tilde\zeta_r(B):B\in \cC^d_0\}
\overset{d}{\longrightarrow} \{\xi(B):B\in \cC^d_0\} \quad \text{as $r\to\infty$}
\end{align*}
in the sense of finite-dimensional distributions.
\end{corollary}
\begin{proof}
The special structure \eqref{Lambda} of the intensity measure $\lambda$ implies the well-known scaling property
\begin{align*}
\BP(\eta_t\in\cdot)=\BP(t^{-1/(d-k)}\eta_1\in\cdot),\quad t>0,
\end{align*}
where $a\eta_1:=\{aE:E\in\eta_1\}$ for $a>0$. Since $\psi$ is homogeneous we obtain for all $B\in\cC^d_0$ and $r>0$ that
\begin{align*}
\zeta_1(rB)&=\frac{1}{m!}r^{\alpha}\int \psi(B\cap r^{-1}E_1\cap\dots \cap r^{-1} E_m)\,
\eta^{(m)}_1\big(\dint(E_1,\ldots,E_m)\big)\notag \\
&=\frac{1}{m!}r^{\alpha}\int \psi(B\cap E_1\cap\dots \cap  E_m)\,
\eta^{(m)}_{r^{1/(d-k)}}\big(\dint(E_1,\ldots,E_m)\big),
\end{align*}
where the second identity holds in distribution jointly in $B$. Hence, we can apply Theorem \ref{thmult} with $g_1(t)=\dots=g_\ell(t):=(m!)^{-1}t^{\alpha(d-k)}$ and then replace $t$ by $r^{1/(d-k)}$.
\end{proof}

\begin{remark}\label{rcovflat}\rm It follows from \eqref{cvij} (with $g_i(t)=g_j(t)=1/m!$) that
\begin{align*}
\CV[\zeta_t(A),\zeta_t(B)]=\sum^m_{n=1}\frac{1}{n!((m-n)!)^2}V_t(A,B,n),\quad A,B\in\cC^d_0,
\end{align*}
where
\begin{align*}
V_t&(A,B,n):=t^{2m-n}
\int\left[\int \psi(A\cap E_1\cap\dots \cap E_n\cap E_{n+1}\cap\dots \cap E_m)\,
\lambda^{m-n}\big(\dint(E_{n+1},\ldots,E_m)\big)\right.\\
&\left.\int \psi(B\cap E_1\cap\dots \cap E_n\cap E_{n+1}\cap\dots \cap E_m)\,
\lambda^{m-n}\big(\dint(E_{n+1},\ldots,E_m)\big)\right]\,
\lambda^{n}\big(\dint(E_{1},\ldots,E_n)\big).
\end{align*}
In accordance with Theorem \ref{cltflats} we therefore obtain
\begin{align*}
\lim_{t\to\infty}t^{-(2m-1)}\CV[\zeta_t(A),\zeta_t(B)]=C(A,B).
\end{align*}
\end{remark}

We now present a couple of examples to which Theorem \ref{cltflats}
as well as Corollary \ref{cltflats2} can be applied.

\begin{example}\label{ex1}\rm Assume that $m(d-k)\le d$.
Assume further that $\cC^d_0=\cC^d$
and that $\psi$ is the $(d-m(d-k))$-dimensional
Hausdorff measure on $\R^d$ restricted to $\cC^d$. Then \eqref{homo}
holds with $\alpha=d-m(d-k)$. Assumption \eqref{bounded} holds
because for $\lambda^{m}$-a.e.\ $(E_1,\dots,E_m)\in A(d,k)^m$
the intersection $E_1\cap\dots\cap E_m$ is either empty
or has dimension $d-m(d-k)$. This follows (recursively) from the argument
given in \cite[p.\ 130]{SW08}.
%
\end{example}

\begin{example}\label{ex2}\rm
Assume that $\cC^d_0=\cC^d$ and that $\psi(B)=\I\{B\ne \emptyset\}$.
Then \eqref{homo} holds with $\alpha=0$ while \eqref{bounded} holds  with $c_B=1$.
\end{example}

Examples \ref{ex1} and \ref{ex2} have been studied in \cite{Hein09,HeinSS06} in the case $k=d-1$. Our results add to Theorem 3.1 and Theorem 4.1 in \cite{Hein09} in several ways. While the latter results are multivariate central limit theorems for the $d$ possible values of the number $m$ of intersections but a fixed (convex) test set $B$, we fix $m$ but study $\zeta_t(B)$ (respectively $\zeta_1(rB)$) as a function of $B$. Further we consider processes of flats and not only hyperplanes. Moreover we obtain  Berry-Esseen-type bounds on the distance $d_3$ and can allow for a considerably larger class of functionals $\psi$. It is also possible to apply Theorem \ref{thmult} to the vector-valued processes arising by varying $m$. This would constitute a complete generalization of \cite{Hein09}. In order  to avoid heavy notation we have refrained from doing so.

We continue with further examples of functionals $\psi$ satisfying \eqref{bounded} and \eqref{homo}. The {\em convex ring} $\cR^d$ is the system of all (possible empty) unions of convex and compact subsets of $\R^d$.

\begin{example}\label{ex3}\rm
Assume that $\cC^d_0=\cR^d$ and that $\psi$ is the {\em intrinsic volume} $V_\alpha$, where $\alpha\in\{0,\dots,d\}$, see e.g.\ \cite{SW08}. Then $\eqref{homo}$ holds. Assumption \eqref{bounded} follows from the fact that $V_\alpha(B\cap E)\le V_\alpha(B)$
for any convex and compact $B\subset\R^d$ and any affine subspace $E\subset\R^d$. By additivity of $V_\alpha$ (see e.g.\ \cite[Section 14.2]{SW08}) the inequality \eqref{bounded} can be established for the whole convex ring.
\end{example}

In contrast to the previous examples, the next functionals are not additive.

\begin{example}\label{ex4}\rm
Assume that $\cC^d_0=\cR^d$ and $\alpha\in\{0,\dots,d-1\}$. Let $\Theta_\alpha(A,\cdot)$ be the {\em support measure} of $A\in \cR^d$, see \cite[Section 14.2]{SW08}. This is a signed measure on the product of $\R^d$ and the unit sphere $\mathbb{S}^{d-1}$ such that
$\Theta_\alpha(A,\R^d\times \mathbb{S}^{d-1})=V_\alpha(A)$. Fix a measurable set $U\subset \mathbb{S}^{d-1}$ and assume that
$$
\psi(A)=\int \I\{(x,u)\in N(A), u\in U\}\,\Theta_\alpha\big(A,\dint(x,u)\big),
$$
where $N(A)$ is the {\em unit normal bundle} of $A$. This consists of all pairs $(x,u)\in\R^d\times \mathbb{S}^{d-1}$ that occur as unique nearest point and associated direction of a point in the complement of $A$, see \cite{HuLaWeil04}. The homogeneity \eqref{homo} follows from the homogeneity of the (non-negative) measure $\I\{(x,u)\in N(A)\}\Theta_\alpha(A,d(x,u))$, see \cite[Proposition 4.9]{HuLaWeil04}. Assumption \eqref{bounded} follows similarly as in Example \ref{ex3} from the additivity of $\Theta_\alpha(A,\cdot)$ in $A\in\cR^d$.
\end{example}

\begin{example}\label{ex5}\rm
Consider the case where $\cC_0^d$ is the space of compact convex subsets of $\R^d$, fix $\alpha\in\{0,\ldots,d\}$ and $\beta\geq 0$ and let $\psi$ be $V_\alpha^\beta$, the power $\beta$ of the intrinsic volume of order $\alpha$. In the case $\alpha=1$ and $\beta=n\in\N$, $\psi$ corresponds to the $n$-th chord-power integral, which is frequently studied in integral geometry, cf.\ \cite[Chapter 8.6]{SW08}. Clearly, \eqref{homo} is satisfied with $\alpha\beta$ there and assumption \eqref{bounded} follows as in
Example \ref{ex3} from the fact that $V_\alpha^\beta(B\cap E)\leq V_\alpha^\beta(B)$ for any convex and compact $B\subset\R^d$ and any affine subspace $E\subset\R^d$.
\end{example}

\begin{remark}\label{r4}\rm If $\psi\ge 0$, then $\CV[\zeta_t(A),\zeta_t(B)]\ge 0$
and $C(A,B)\ge 0$ for all $A,B\in\cC^d_0$. This is the case in Examples \ref{ex1},
\ref{ex2}, and \ref{ex5}. Taking as $\cC^d_0$ the system of convex sets,
this is also the case in Examples \ref{ex3} and \ref{ex4}.
If additionally $m\geq 2$, Corollary \ref{c2a} shows
that the moment generating functions of the functionals under consideration
do not exist.
\end{remark}

\subsection*{Acknowledgement}
The authors would like to thank two anonymous referees for a number of valuable comments, which were helpful for us to improve the text.


\begin{thebibliography}{99}

\bibitem{BhattacharyaGhosh}
Bhattacharya, R.N.\ and Ghosh, J.K.\ (1992).
A class of U-statistics and asymptotic normality of the number of $k$-clusters.
{\em J. Multivariate Anal.} {\bf 43}, 300-330.

\bibitem{Bil}
Billingsley, P.\ (1979).
\newblock {\it Probability and Measure}.
\newblock Wiley, New York.

\bibitem{Decreusefond}
Decreusefond, L., Ferraz, E., Randriam, H.\ and Vergne, A.\ (2011).
Simplicial homology of random configurations.
{\em arXiv:} 1103.4457 [math.PR].

\bibitem{Hein09}
Heinrich L.\ (2009). Central limit theorems for motion-invariant Poisson
hyperplanes in expanding convex windows.
{\em Rendiconti del Circolo Matematico di Palermo Series II},
Suppl.\ {\bf 81}, 187-212.

\bibitem{HeinSS06}
Heinrich, L., Schmidt, H.\ and Schmidt, V.\ (2006).
Central limit theorems for Poisson hyperplane tessellations.
{\em Ann. Appl. Probab.} {\bf 16}, 919-950.

\bibitem{HuLaWeil04}
Hug, D., Last, G.\  and Weil, W. (2004).
A local Steiner--type formula for general closed sets and applications.
{\em Mathematische Zeitschrift}  {\bf 246}, 237-272.

\bibitem{Kallenberg}
Kallenberg, O.\ (2002).
\newblock {\it Foundations of Modern Probability}.
\newblock Second Edition, Springer, New York.

\bibitem{LRPeccati1}
Lachi\`eze-Rey, R.\ and Peccati, G.\ (2013).
Fine Gaussian fluctuations on the Poisson space, I: contractions, cumulants and geometric random graphs.
Electron. J. Probab. \textbf{18}, Article 32.

\bibitem{LRPeccati2}
Lachi\`eze-Rey, R.\ and Peccati, G.\ (2013).
Fine Gaussian fluctuations on the Poisson space, II: rescaled kernels, marked processes and geometric U-statistics.
To appear in \textit{Stoch. Proc. Appl.}, DOI: 10.1016/j.spa.2013-06-004.


\bibitem{LaPe11}
Last, G.\ and Penrose, M.D.\ (2011).
Fock space representation, chaos expansion and covariance inequalities
for general Poisson processes.
{\it Probab.\ Theory and Related Fields} {\bf 150}, 663-690.

\bibitem{Minh}
Minh, N.T.\ (2011).
Malliavin-Stein method for multi-dimensional U-statistics of Poisson point processes.
{\em arXiv:} 1111.2140 [math.PR].

\bibitem{PSTU10}
Peccati, G., Sol\'e, J. L., Taqqu, M.S.\  and Utzet, F.\ (2010).
Stein's method and normal approximation of Poisson functionals.
{\em Ann. Probab.} {\bf 38}, 443-478.

\bibitem{TaPecc09}
Peccati, G.\ and Taqqu, M.S.\  (2011).
{\em Wiener Chaos: Moments, Cumulants and Diagrams.}
Springer, Milan.

\bibitem{PeZheng10}
Peccati, G.\ and Zheng, C.\ (2010).
Multi-dimensional Gaussian fluctuations on the Poisson space.
{\em Elec. J. Probab.} {\bf 15}, 1487-1527.

\bibitem{Penrose}
Penrose, M.D.\ (2003).
\newblock {\em Random Geometric Graphs}.
\newblock Oxford University Press, Oxford.

\bibitem{ReitznerSchulte2011}
Reitzner, M.\ and Schulte, M.\ (2012+).
Central limit theorems for U-statistics of Poisson point processes.
To appear in {\em Ann. Probab.}

\bibitem{SW08}
Schneider, R.\ and  Weil, W.\ (2008).
\newblock {\em Stochastic and Integral Geometry}.
\newblock Springer, Berlin.

\bibitem{ST13}
Schulte, M.\ and Th\"ale, C.\ (2012+).
Distances between Poisson $k$-flats.
To appear in {\em Methodol. Comput. Appl. Probab.}, DOI 10.1007/s11009-012-9319-2.

\bibitem{Surg84}
Surgailis, D. (1984).
On multiple Poisson stochastic integrals and associated Markov semigroups.
{\it Probab. Math. Statist.} {\bf 3}, 217-239.


\end{thebibliography}
\end{document}